\documentclass[a4paper,12pt,reqno]{amsart}

\usepackage[letterpaper, hmargin=1in, top=1in, bottom=1.2in, footskip=0.7in]{geometry}
\usepackage{cancel} 

\usepackage{rotating,pdflscape,xcolor,amssymb,subfigure,psfrag,
amsmath,eufrak,bbm,epsfig,amsthm,mathtools,fancyhdr,graphicx}
\definecolor{vdarkred}{rgb}{0.6,0,0.2}
\definecolor{vdarkblue}{rgb}{0,0.2,0.6}
\usepackage[pdftex, colorlinks, linkcolor=vdarkblue,citecolor=vdarkred]{hyperref}
\usepackage{a4wide,amscd}
\usepackage{tikz} 
 \usepackage[usenames,dvipsnames]{pstricks}
\usepackage{placeins}

\usepackage[small]{caption}
\usepackage[normalem]{ulem} 

\newcommand{\om}{\omega}

\newcommand{\beg}{\begin}
\newcommand{\en}{\end}

\renewcommand{\Im}{\mathfrak{Im}}

\newcommand{\bgt}{\begin{itemize}}
\newcommand{\ent}{\end{itemize}}

\newcommand{\la}{\label}

\newcommand{\brem}{\begin{rmk}}
\newcommand{\erem}{\end{rmk}}
\newcommand{\blem}{\begin{lem}}
\newcommand{\elem}{\end{lem}}
\newcommand{\bcor}{\begin{cor}}
\newcommand{\ecor}{\end{cor}}
\newcommand{\bTh}{\begin{Th}}
\newcommand{\eTh}{\end{Th}}
\newcommand{\bpropo}{\begin{propo}}
\newcommand{\epropo}{\end{propo}}

\newcommand{\op}{\operatorname} 

\newcommand{\Var}{\operatorname{Var}}

\newcommand{\diag}{\operatorname{diag}}

\renewcommand{\P}{\mathbb{P}}

\newcommand{\Tr}{\operatorname{Tr}}

\newcommand{\E}{\op{\mathbb{E}}}

\newcommand{\R}{\mathbb{R}}
\newcommand{\C}{\mathbb{C}}

\newcommand{\eps}{\varepsilon}

\newcommand{\bbm}{\begin{bmatrix}}
\newcommand{\ebm}{\end{bmatrix}}
\newcommand{\bes}{\begin{equation*}}
\newcommand{\ees}{\end{equation*}}
\newcommand{\be}{\begin{equation}}
\newcommand{\ee}{\end{equation}}
\newcommand{\beqy}{\begin{eqnarray}}
\newcommand{\eeqy}{\end{eqnarray}}
\newcommand{\beq}{\begin{eqnarray*}}
\newcommand{\eeq}{\end{eqnarray*}}

\newcommand{\bpm}{\begin{pmatrix}}
\newcommand{\epm}{\end{pmatrix}}

\newcommand{\bpr}{\beg{proof}}
\newcommand{\epr}{\en{proof}}

\newcommand{\begenum}{\begin{enumerate}}
\newcommand{\enenum}{\end{enumerate}}

\newtheorem{Th}{Theorem}[section]

\newtheorem{propo}[Th]{Proposition}
\newtheorem{proposition}[Th]{Proposition} 
 
\newtheorem{lem}[Th]{Lemma}

\newtheorem{cor}[Th]{Corollary}
\theoremstyle{definition}
\newtheorem{rmk}[Th]{Remark}

\long\def\symbolfootnote[#1]#2{\begingroup
\def\thefootnote{\fnsymbol{footnote}}\footnote[#1]{#2}\endgroup}

\usepackage{listings}
\author[J.-P.\ Bouchaud]{Jean-Philippe Bouchaud}
\address{Jean-Philippe Bouchaud:  Académie des Sciences, 23 Quai de Conti, 75006 Paris, France and Capital Fund Management, 23 rue de l'Universit\'e, 75007 Paris, France}
\email{Jean-Philippe.Bouchaud@cfm.com}

\author[P. Bousseyroux]{Pierre Bousseyroux} \address{Pierre Bousseyroux: Econophysics Lab, Institut Louis Bachelier, 28 Pl. de la Bourse, Palais Brongniart, 75002 Paris, France and LadHyX, UMR CNRS 7646, École Polytechnique, Institut Polytechnique de Paris, 91128 Palaiseau, France} \email{pierre.bousseyroux@polytechnique.edu}
\author[T. Espana]{Tomas Espana}
\address{Tomas Espana: Department of Operations Research and Financial Engineering, Princeton University, Princeton, NJ, 08540, USA}
\email{tomas.espana@princeton.edu}

\author[M. Smerlak]{Matteo Smerlak}
\address{Matteo Smerlak: Capital Fund Management, 23 rue de l'Universit\'e, 75007 Paris, France}
\email{matteo.smerlak@cfm.com}

\date{\today}
\subjclass[2020]{60B20, 62H20} 
\keywords{Spearman correlation matrices, scale-mixture model, Marčenko-Pastur law}

\title[Spearman Spectra under Scale Mixtures]{Spectra of high-dimensional Spearman correlation matrices under scale-mixture dependence}
\begin{document}
\maketitle
\begin{abstract}
We study the asymptotic spectral properties of high-dimensional Spearman correlation matrices for scale-mixture data. We consider observations of the form $x_t=\sigma_t \xi_t \in \R^N,$ where the coordinates of $\xi_t$ are i.i.d.\ and the scalar mixture variable $\sigma_t$ is shared by all coordinates. Under natural symmetry assumptions, the coordinates of $x_t$ are pairwise uncorrelated in both the Pearson and Spearman sense. Nevertheless, they are not independent when the mixture variable is non-degenerate. We show that this higher-order dependence survives the rank transformation and leaves a nontrivial spectral signature. In the proportional regime $N/T\to q\in(0,\infty),$ the empirical spectral distribution of the Spearman correlation matrix converges almost surely to a generalized Marčenko--Pastur law governed by the limiting distribution of an effective rank variance. We also formulate a broader latent-variable extension, which covers, in particular, some scale-mixture models with correlated directional components. We discuss solvable examples and numerical approximations, motivated in part by heavy-tailed data in robust multivariate statistics, econometrics, and finance.
\end{abstract}


\section{Introduction}

Statistical inference for covariance and correlation matrices is a central problem in high-dimensional multivariate analysis \cite{cai2017global, fan2016overview}. When the dimension $N$ is of the same order as the sample size $T$, classical fixed-dimensional intuition no longer applies: estimation errors in individual entries may accumulate and produce nontrivial spectral effects. Random matrix theory provides a precise framework to understand this phenomenon. In the simplest setting of independent coordinates with identity covariance, the empirical spectral distribution of the sample covariance matrix converges to the Marčenko--Pastur law \cite{marchenko1967distribution}. More generally, when the population covariance is nontrivial, the limiting spectrum is described by generalized Marčenko--Pastur laws \cite{silverstein1995empirical, silverstein1995analysis, silverstein1995strong}.

Rank-based correlation matrices, such as Spearman's rho \cite{spearman1987proof} and Kendall's tau \cite{kendall1938new}, have attracted increasing attention in high-dimensional statistics. Since they depend only on the relative order of the observations, rank correlations are invariant under monotone marginal transformations and are less sensitive to heavy tails and outliers than moment-based quantities. This makes them particularly appealing in robust multivariate statistics and for data sets where second moments may be unstable or even infinite \cite{liu2012high, han2014scale, han2016statistical}. They also provide a natural nonparametric alternative to Pearson correlation matrices while retaining a matrix structure suitable for spectral analysis. In random matrix theory, the limiting spectral distribution (LSD) of Spearman correlation matrices was first studied in \cite{bai2008large} under i.i.d.\ features and later extended to dependent features in \cite{wu2022limiting}, where the LSD is shown to be a generalized Marčenko--Pastur law. More recently, \cite{dornemann2025ties} studied the effect of ties on the spectra of rank-based dependence matrices. Beyond global spectral limits, \cite{bao2019tracy} proved Tracy--Widom fluctuations for the largest eigenvalues of Spearman matrices, and central limit theorems for linear spectral statistics were obtained in \cite{bao-clt-spearman-2015, chen2024large}.

The present paper contributes to this line of work by focusing on a concrete and interpretable dependence mechanism, namely scale mixtures. While \cite{wu2022limiting} provides a general approach to LSDs of Spearman matrices under dependent features, here we compute explicitly the effective rank variance generated by a common latent scale and analyze the resulting generalized Mar\v{c}enko--Pastur law through solvable examples and numerical approximations. We consider observations of the form $x_t=\sigma_t \xi_t \in \R^N,$ where the coordinates of $\xi_t$ are i.i.d.\ and the scalar mixture variable $\sigma_t$ is shared by all coordinates. Such models provide a simple mechanism for dependence through a common latent scale. They are also a natural class in multivariate statistics: when $\xi_t$ is Gaussian and $\sigma_t^2$ is inverse-gamma distributed, $x_t$ has a multivariate Student distribution; more generally, elliptical distributions are closely related to scale mixtures of normal distributions \cite{gomez2006sequences}. Because Spearman matrices depend only on coordinatewise ranks, the same conclusions automatically hold after arbitrary strictly increasing marginal transformations, connecting the model to the transelliptical framework studied in high-dimensional semiparametric statistics \cite{liu2012transelliptical,han2014scale}. Thus, even in the identity-scale case considered here, the model captures a broad class of (heavy-tailed) distributions for which rank-based methods are especially natural. 

The scale-mixture model also has a natural interpretation in econometrics and finance. In the context of asset returns, the variable $\sigma_t$ can be viewed as a systemic volatility factor affecting the whole market at time $t$, while $\xi_t$ describes the idiosyncratic directional component of returns \cite{clark1973subordinated,epps1976stochastic}. Empirical studies further indicate that financial returns and volatility display heavy-tailed behavior \cite{mandelbrot1963variation, liu1999statistical, cont2001empirical}. The same scale-mixture model was studied in \cite{BiroliBouchaudPotters2007} to determine the spectra of Pearson covariance and correlation matrices for Student ensembles. In that setting, the long-tailed fluctuations of the scale variable $\sigma_t$ generate a limiting eigenvalue density with no bounded right edge. This provides an additional motivation for considering rank-based correlation matrices such as Spearman when studying heavy-tailed distributions.

The paper is built around two limiting spectral distribution theorems. First, our main result is Theorem \ref{thm:main}, which is stated for the scale-mixture model with i.i.d.\ directional components. It proves that, in the proportional regime $N/T\to q\in(0,\infty)$, the empirical spectral distribution of the sample Spearman correlation matrix converges almost surely to a deterministic generalized Marčenko--Pastur law. The limiting law is governed by the distribution of an effective rank variance induced by the mixture variable. In particular, the classical Marčenko--Pastur law is recovered only in the degenerate case where this effective variance is deterministic. Thus, the Spearman rank transformation does not wash out the mixture variable: even though magnitudes are discarded, the latent scale remains visible at the spectral level. Second, Theorem \ref{thm:generalized_main} shows that the same mechanism extends to a broader class of latent-variable dependence models, which contains the scale-mixture model as a special case and also includes regime-switching models and one-factor Gaussian scale-mixture models with correlated directional components.

The paper is organized as follows. Section \ref{sec:model-result} introduces the Spearman correlation matrix, the scale-mixture model, and Theorem \ref{thm:main}. Section \ref{sec:solvable_models} presents solvable examples and numerical approximations, including the beta-copula approximation for Student data. Section \ref{sec:gene_model} states the latent-variable extension, Theorem \ref{thm:generalized_main}, and discusses examples including regime-switching and one-factor Gaussian models. Section \ref{sec:proof} contains the technical proofs. Although Theorem \ref{thm:main} is covered by the more general Theorem \ref{thm:generalized_main}, we give the full proof of Theorem \ref{thm:main} first, since the mechanism is most transparent in this setting. The proof of Theorem \ref{thm:generalized_main} is then obtained by adapting the same argument. The Appendix contains the remaining proofs and complementary technical results.

\section{Scale-Mixture Model and Main Result}\la{sec:model-result}

\subsection{Scale-Mixture Model}\la{sec:scale-mixture}

For each $T \geq 1$, let $N=N(T)$ and assume that:
\begin{enumerate}
    \item[(i)] $(\sigma_t)_{1\leq t\leq T}$ are i.i.d.\ copies of an almost surely positive random variable $\sigma$,
    \item[(ii)] $(\xi_{t,n})_{1\leq t\leq T,\, 1\leq n\leq N}$ are i.i.d.\ copies of a real random variable $\xi$ symmetric about $0$,
    \item[(iii)] the families $(\sigma_t)_t$ and $(\xi_{t,n})_{t,n}$ are independent,
    \item[(iv)] the distribution function $F$ of $\sigma\xi$ is continuous. 
\end{enumerate}
For each $1\leq t \leq T$, we consider the observations \be\la{eq:scale-mixture-model}x_t = \sigma_t \xi_t, \qquad \xi_t = (\xi_{t,1},\ldots,\xi_{t,N})^\top.\ee
The scalar $\sigma_t$ is a mixture variable shared by all coordinates of the observation $x_t$, while $\xi_t$ contains the idiosyncratic components. Thus the model introduces dependence across coordinates through a common latent scale. By construction, the vectors $(x_t)_{t\geq 1}$ are i.i.d.; for each fixed $t$, the coordinates $(x_{t,n})_{n\geq 1}$ are independent conditionally on $\sigma_t$, but are generally not independent unconditionally. When the relevant second moments exist, the symmetry assumption implies that the coordinates are pairwise uncorrelated in the Pearson sense (and also in the Spearman sense, see Section \ref{sec:spearman-def}). Note, however, that no moment assumption on the variables $x_{t,n}$ is required for our main theorem.

\subsection{Spearman Correlation Matrix}\la{sec:spearman-def}

For each coordinate $1\leq n\leq N$ and each observation $1\leq t\leq T$, let 
$$
Q_{n,t}
=
\sum_{s=1}^T \mathbf 1_{\{x_{s,n}\leq x_{t,n}\}}
$$
be the rank of $x_{t,n}$ among $x_{1,n},\ldots,x_{T,n}$. Since the common marginal distribution function $F$ is continuous, there are almost surely no ties. The sample Spearman rank correlation matrix is defined by
\be
R_N=ZZ^\top,
\qquad
Z_{n,t}
=
\frac{Q_{n,t}-\frac{T+1}{2}}
{\sqrt{\frac{T(T^2-1)}{12}}}.
\la{eq:sample-spearman-matrix}
\ee
The $(k,l)$-entry of $R_N$ is the usual sample Spearman rank correlation coefficient between the $k$-th and $l$-th coordinates, and the diagonal entries of $R_N$ are equal to one. Our goal is to determine the limiting spectral distribution of $R_N$ under the scale-mixture model \eqref{eq:scale-mixture-model}.

\brem\la{rmk:condi_spearman_moments}
It is useful to compare $R_N$ with its population counterpart. Define the standardized population rank transform $g(x)=\sqrt{3}\bigl(2F(x)-1\bigr), \, x\in\R.$ Since $F$ is continuous, $F(x_{t,n})$ is uniformly distributed on $(0,1)$, so $\E[g(x_{t,n})]=0$ and $\E[g(x_{t,n})^2]=1.$
The population Spearman correlation matrix is then
$$
\Sigma_N^{\mathrm S}
=
\left(
\E\!\left[g(x_{t,n})g(x_{t,m})\right]
\right)_{1\leq n,m\leq N}.
$$
In the scale-mixture model, this matrix is the identity. Indeed, $x_{t,n}$ is symmetric about 0 and by continuity of $F,$ we have $F(-x)=1-F(x)$ for any $x$ and therefore $g$ is odd. Hence, for $n\neq m$,
$$
\begin{aligned}
\E\!\left[g(x_{t,n})g(x_{t,m})\mid \sigma_t\right]
&=
\E\!\left[g(\sigma_t\xi_{t,n})\mid\sigma_t\right]
\E\!\left[g(\sigma_t\xi_{t,m})\mid\sigma_t\right] =0,
\end{aligned}
$$
where we used the conditional independence of $x_{t,n}$ and $x_{t,m}$ given $\sigma_t$. This yields $\Sigma_N^{\mathrm S}=I_N.$
\erem

\brem
Since Spearman correlations depend only on ranks, they are invariant under coordinatewise strictly increasing transformations. More precisely, if $h_1,\ldots,h_N:\mathbb R\to\mathbb R$ are strictly increasing and $y_{t,n}=h_n(x_{t,n})=h_n(\sigma_t\xi_{t,n}),$ then the Spearman matrices computed from $y$ and $x$ coincide, $R_N(y)=R_N(x).$ Thus the limiting spectral distribution obtained in Theorem \ref{thm:main} also applies to this transformed scale-mixture model. In particular, when the latent vector $x_t=\sigma_t\xi_t$ is elliptical, as in the Gaussian scale-mixture case, the transformed observations $y_t$ belong to the transelliptical family \cite{liu2012transelliptical,han2014scale}.
\erem

\subsection{Main Theorem}

We now state the LSD of the Spearman matrix $R_N$ defined in \eqref{eq:sample-spearman-matrix} under the scale-mixture model \eqref{eq:scale-mixture-model} and introduce some notation. We denote $\C^+ = \{z \in \C, \Im z > 0\}$ and $\C^- = \{z \in \C, \Im z < 0\}$. For a probability measure $\mu$ on $\R$, we denote its Stieltjes transform $m_\mu(z) = \int \!\frac{1}{z-x}d\mu$ for $z\in\C^+.$ For an $N\times N$ Hermitian matrix $A$ with eigenvalues $\lambda_1\geq\ldots\geq\lambda_N$, we denote its empirical spectral distribution by
$$
F^A=\frac{1}{N}\sum_{i=1}^N \delta_{\lambda_i}.
$$
For $a>0$, define the effective rank variance $s(a)=\E\!\left[g(a\xi)^2\right]\!.$ Let $\mu_s$ denote the law of $s(\sigma)$, i.e., 
$$
\mu_s = \operatorname{Law}(s(\sigma)),
$$
which is supported on $[0,3],$ since $|g|\leq \sqrt{3}$.

\bTh
\la{thm:main}
Suppose that {\rm (i)--(iv)} hold. Assume moreover that $\frac{N}{T}\rightarrow q\in(0,\infty)$ and that $s(\sigma)>0$ almost surely. Then the empirical spectral distribution of the sample Spearman rank correlation matrix $R_N$ converges weakly almost surely to a deterministic probability measure $F_{q,\mu_s}$ on $[0,\infty)$. Moreover, its Stieltjes transform $m(\cdot)$ is the unique function $m:\C^+\to \C^-$ satisfying
\be
z =
\frac{1}{m(z)}
+
\int_{[0,3]}
\frac{x}{1-qxm(z)}\,d\mu_s(x),
\qquad z\in\C^+.
\la{eq:main-fixed-point}
\ee
\eTh

\brem
The probability measure $F_{q,\mu_s}$ is a generalized Marčenko--Pastur (MP) law. When $\sigma$ is deterministic it follows that $\mu_s=\delta_1$ so that \eqref{eq:main-fixed-point} reduces to the classical MP equation. In this case, the LSD of the Spearman matrix is the MP law, as first established in \cite{bai2008large}. 
\erem

\brem
Recall that the $R$-transform of a compactly supported probability measure
$\mu$ is defined by
$$
R_\mu(w)=K_\mu(w)-\frac{1}{w},
$$
where $K_\mu$ is the functional inverse of its Stieltjes transform
$m_\mu$ near $w=0$. Hence, setting $w=m_\mu(z)$ in
\eqref{eq:main-fixed-point}, we obtain
$$
K_{F_{q,\mu_s}}\!(w)=
\frac{1}{w}
+
\int_{[0,3]}
\frac{x}{1-qxw}\,d\mu_s(x).
$$
Consequently, the $R$-transform of $F_{q,\mu_s}$ is
\be
R_{F_{q,\mu_s}}\!(w)
=
\int_{[0,3]}
\frac{x}{1-qxw}\,d\mu_s(x) = \E\!\left[\frac{s(\sigma)}{1-qws(\sigma)}\right],
\la{eq:r_transform}
\ee
for $w$ in a neighborhood of $0$. In particular, when $\mu_s=\delta_1$, this
gives
$$
R_{F_{q,\delta_1}}\!(w)=\frac{1}{1-qw},
$$
which is the $R$-transform of the classical Marčenko--Pastur law with
aspect ratio $q$.
\erem

\begin{cor}\la{coro:moment-proof}
Let $F_{q,\mu_s}$ be the LSD in Theorem \ref{thm:main}. Then $F_{q,\mu_s}$ is compactly supported and has an atom at zero of mass $\left(1-\frac1q\right)_+$. Its first two moments are
$$
    \int x\,dF_{q,\mu_s}(x)=1 \qquad \text{and} \qquad
    \int x^2\,dF_{q,\mu_s}(x)
    =
    1+q+q\operatorname{Var}(s(\sigma)).
$$
Moreover, the free cumulants are given by $\kappa_k = q^{k-1}\E[s(\sigma)^k]$ for $k\geq 1$.
\end{cor}

\brem
The corollary gives the usual atom at the origin for the classical Marčenko–Pastur law: since $s(\sigma) >0$ almost surely, no additional atom is contributed by the limiting weight distribution $\mu_s$ and the only possible atom at zero is the dimensional one, of mass $\left(1 -\frac1q\right)_+$.
\erem
\brem
The corollary shows that the first moment agrees with that of the classical Marčenko--Pastur law with aspect ratio $q$, whereas the second moment is larger by $q\operatorname{Var}(s(\sigma)).$ Consequently, whenever $s(\sigma)$ is nonconstant almost surely, the mixture variable strictly increases the limiting spectral variance.
\erem
\bpr
See Appendix \ref{app:moment-proof}.
\epr

\section{Solvable Models and Numerical Illustrations}\la{sec:solvable_models}
The fixed-point equation of Theorem \ref{thm:main} is explicit once the law of the effective rank variance $s(\sigma)$ is known. In this section, we discuss solvable families for which the corresponding $R$-transform can be written in closed form, and compare the resulting limiting spectra with finite-size simulations and with the classical Marčenko--Pastur benchmark. In what follows, we denote by $\mathrm{Beta}(\gamma,\eta)$ the beta distribution on $(0,1)$ with density proportional to $x^{\gamma-1}(1-x)^{\eta-1}$ with $\gamma>0, \eta>0$.

\subsection{Binary-spin model}

We first consider a minimal solvable case in which the directional variables $(\xi_{t,n})_{t,n}$ are independent Rademacher variables. The resulting limiting law is compared with simulations and with the classical Marčenko--Pastur benchmark in Fig.\ \ref{fig:binary-spin}.

\begin{proposition}\la{prop:binary-spin}
Assume that $\xi$ is Rademacher and that $\sigma$ has a continuous distribution, is almost surely positive and independent of $\xi$.  Then,
$$
        \frac{s(\sigma)}{3} \sim \mathrm{Beta}\left(\frac12,1\right).
$$
In particular, the limiting spectral distribution in Theorem \ref{thm:main} has
$R$-transform
$$
        R_{F_{q,\mu_s}}\!(w)
        =
        \frac{1}{qw}
        \left(
        \frac{\operatorname{arctanh}\sqrt{3qw}}{\sqrt{3qw}}
        -1
        \right),
$$
for $w$ in a neighborhood of zero.
\end{proposition}

\brem
Note that under the binary-spin model the distribution of $s(\sigma)/3$ is universal in the sense that it does not depend on the distribution of the mixture variable $\sigma$.
\erem

\bpr
See Appendix \ref{app:proof-binary-spin}
\epr

\begin{figure}[h!]
    \centering
    \includegraphics[width=0.9\textwidth]{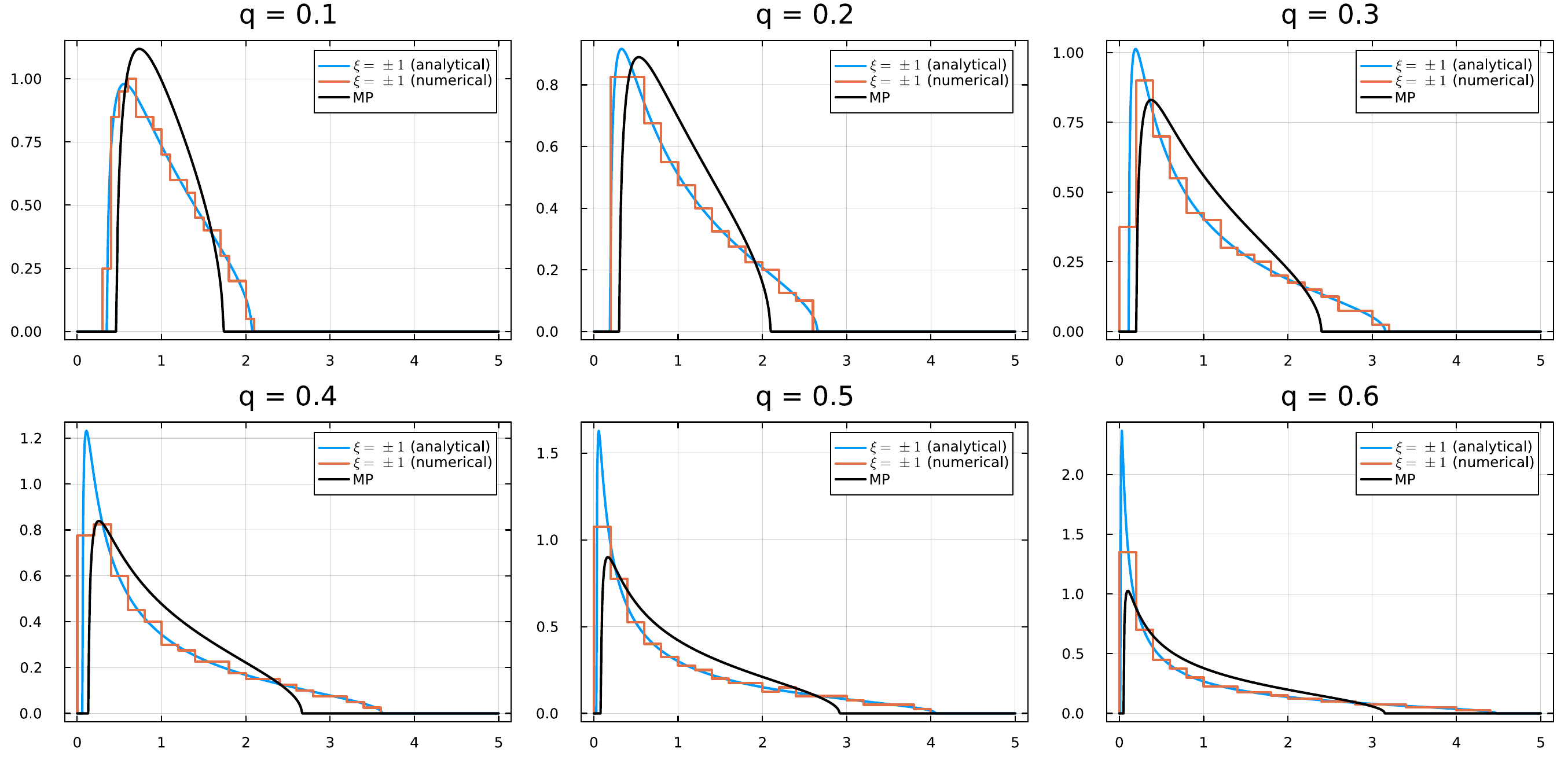}
    \caption{Binary-spin model. The orange histograms show the empirical eigenvalue distributions of $R_N$ obtained from simulations with $N=200$ and several values of the aspect ratio $q=N/T$. The blue curves show the analytical limiting density derived from the explicit $R$-transform of Proposition \ref{prop:binary-spin}, and the black curves show the classical Marčenko--Pastur density with aspect ratio $q$.}
    \label{fig:binary-spin}
\end{figure}

\FloatBarrier
\subsection{Beta-copula and one-cut approximations with application to Student data}
\label{subsec:beta-copulas}

\subsubsection{Beta-copula approximation}
The binary-spin model gives a solvable example in which the effective rank variance has a beta distribution. This motivates the following more general family of benchmarks. We henceforth say that the model has a beta-copula with parameter $\alpha>0$ if the effective rank variance $s(\sigma)$ satisfies
$$
        \frac{s(\sigma)}{3}\sim \operatorname{Beta}(\alpha,2\alpha).
$$
The parameters are chosen so that $\mathbb E[s(\sigma)]=1$, consistent with the identity
$\mathbb E[s(\sigma)]=\mathbb E[g(\sigma\xi)^2]=1$.

\begin{proposition}
\label{prop:beta-copula}
Assume that
$$
        \frac{s(\sigma)}{3}\sim \operatorname{Beta}(\alpha,2\alpha),
        \qquad \alpha>0.
$$
Then the limiting spectral distribution of Theorem \ref{thm:main} has
$R$-transform
$$
R_{F_{q,\mu_s}}\!(w)={}_2F_1\!\left(1,\alpha+1;3\alpha+1;3qw\right),
$$
for $w$ in a neighborhood of zero, where ${}_2F_1$ denotes the Gauss hypergeometric function.
\end{proposition}

\bpr
See Appendix \ref{app:proof-beta-copula}
\epr

We now use the beta-copula family as a tractable approximation for the spectral law obtained when the observations $x$ follow a multivariate Student distribution. This is a natural class of examples in the present scale-mixture framework. Indeed, Student distributions are canonical elliptical distributions and admit the standard Gaussian scale-mixture representation: given $\Sigma \!\in\! \R^{N\times N}$ and $\nu\!>\!0,$ if $\xi\sim \mathcal N(0,\Sigma)$ and, independently, $\sigma^2 \sim \operatorname{InvGamma}\left(\frac{\nu}{2},\frac{\nu}{2}\right)\!,$ then $x=\sigma \xi$ has a multivariate Student distribution with $\nu$ degrees of freedom and scale matrix $\Sigma$. More generally, elliptical distributions are closely related to scale mixtures of normal distributions; see, for instance, \cite{gomez2006sequences}. Beyond this structural connection, Student-type models are also natural from an applied perspective, as they provide a simple way to model heavy-tailed data arising in robust multivariate statistics, econometrics, and finance. In the numerical experiments below, we focus on the identity-scale case, which belongs to the i.i.d.\ directional framework of Theorem \ref{thm:main}; correlated Gaussian directions are discussed in Section \ref{sec:gene_model}.

To obtain a tractable approximation in this setting, we proceed as follows. For different values of $\nu$, we simulate Student observations with identity scale matrix and estimate the corresponding distribution of the effective rank variance $s(\sigma).$ We then fit a one-parameter beta distribution of type $\textrm{Beta}(\alpha, 2\alpha)$ to $s(\sigma)/3$ by maximum likelihood, and denote the resulting estimate by $\widehat{\alpha}_{\mathrm{MLE}}(\nu)$. The R-transform of the LSD is then deduced by Proposition \ref{prop:beta-copula}. Fig.\ \ref{fig:student} shows that the agreement between the empirical spectrum and the fitted analytical density is very good, thereby supporting the beta-copula approximation. We also observe that as the number of degrees of freedom $\nu$ increases, the classical Marčenko--Pastur law becomes progressively closer to the empirical spectrum. This is expected: as $\nu\to\infty,$ the Student distribution converges to the Gaussian distribution with identity scale, whose coordinates are independent, and the limiting spectral distribution reduces to the classical Marčenko--Pastur law.

\begin{figure}[h!]
    \centering
    \includegraphics[width=0.9\textwidth]{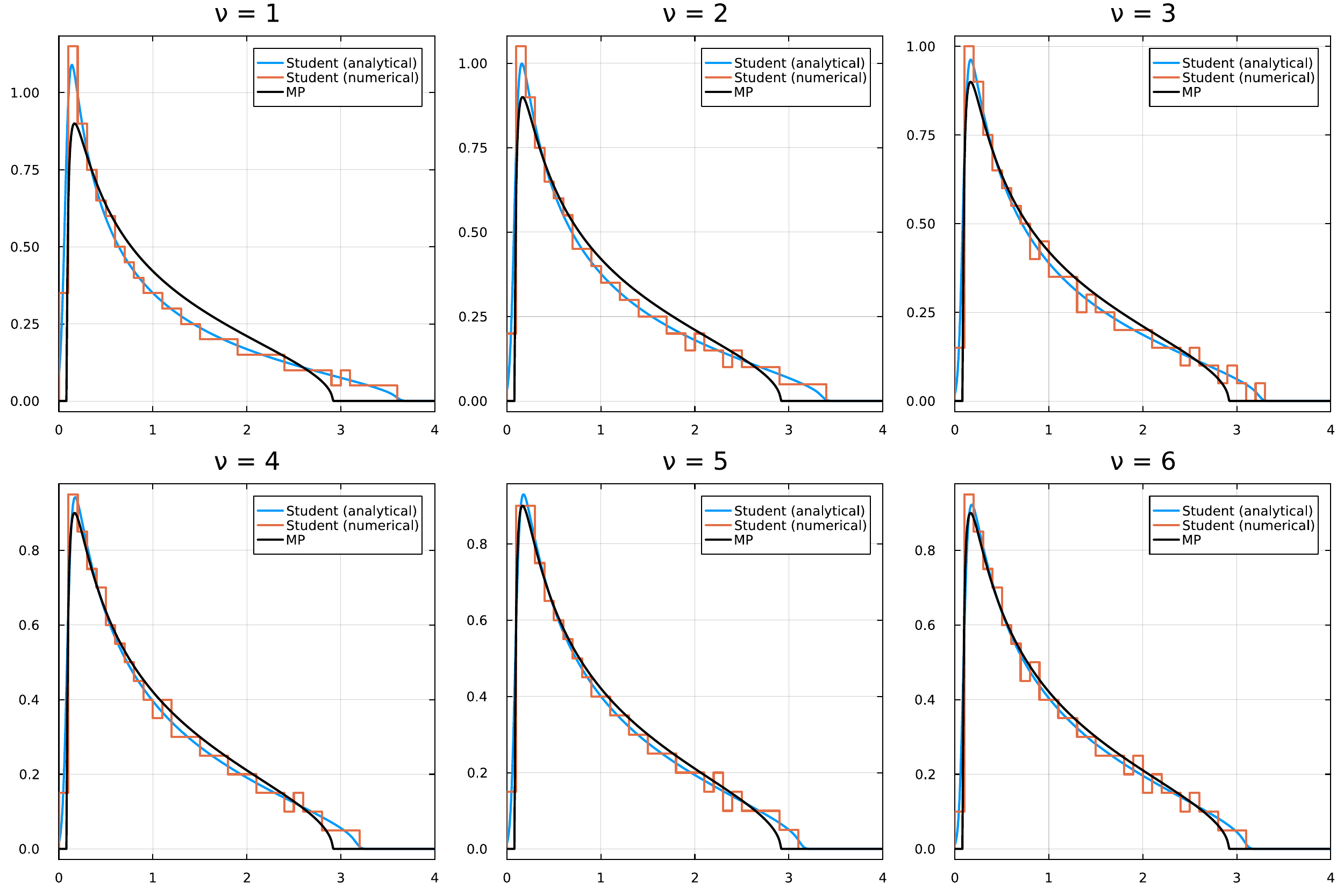}
    \caption{Student data and beta-copula approximation. Empirical results for multivariate Student observations with $N=200$, $q=0.5$ and degrees of freedom $\nu=1,\ldots,6$. The orange histograms show the empirical eigenvalue distributions of $R_N$ obtained from simulations, the blue curves show the analytical densities obtained by fitting a beta-copula law $\operatorname{Beta}(\alpha,2\alpha)$ to the simulated distribution of $s(\sigma)/3$, and the black curves show the classical Marčenko--Pastur density with aspect ratio $q$.}
    \label{fig:student}
\end{figure}

To further describe the beta-copula approximation, Figure \ref{fig:alpha-mle} displays the fitted parameter $\widehat{\alpha}_{\mathrm{MLE}}$ as a function of the Student degrees of freedom $\nu$. Over the range considered, the fitted parameter appears to grow linearly with $\nu$. This increase is consistent with the progressive concentration of the distribution of $s(\sigma)/3$ around $1/3$, since $\mathrm{Beta}(\alpha,2\alpha)$ converges to a point mass at $1/3$ as $\alpha\to\infty$.

\begin{figure}[t]
    \centering
    \includegraphics[width=0.45\textwidth]{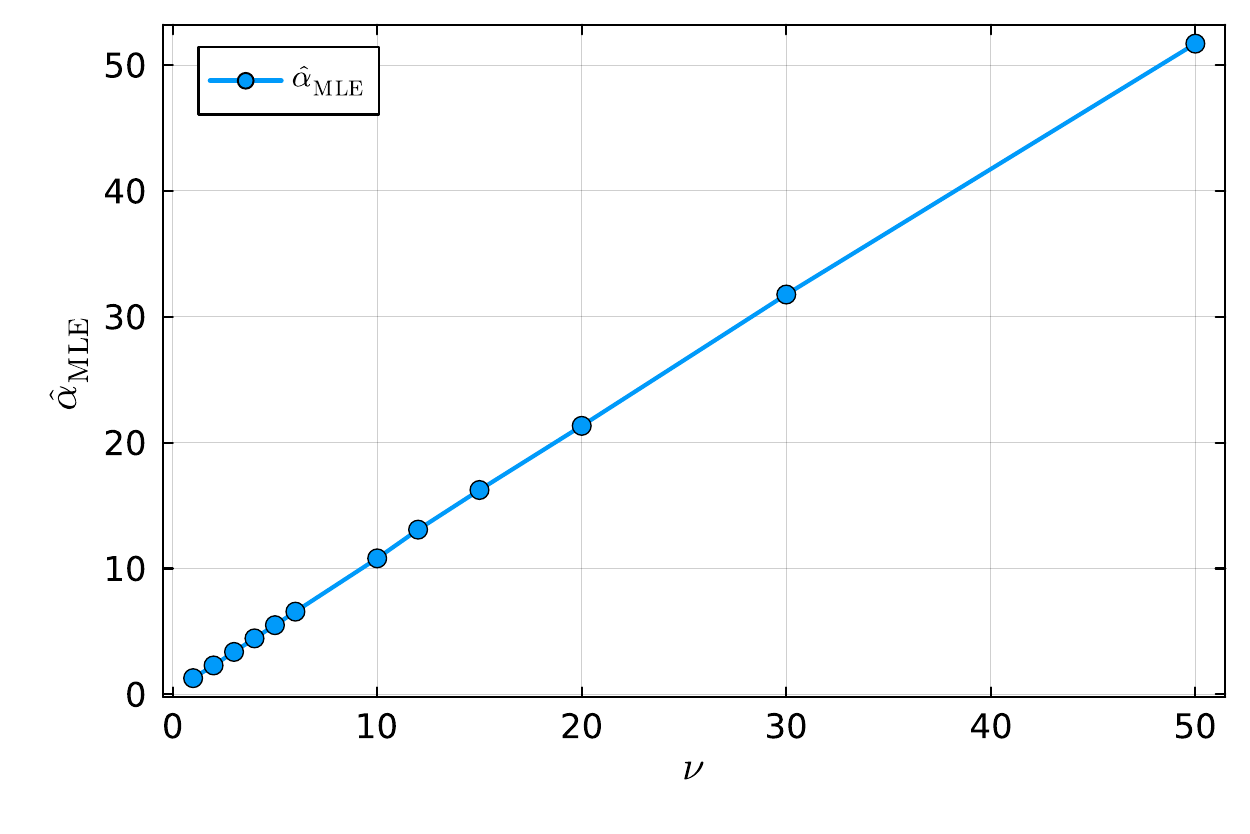}
    \caption{Fitted beta parameter
    $\widehat{\alpha}_{\mathrm{MLE}}$ as a function of the Student degrees
    of freedom $\nu$. For each $\nu$, the parameter is obtained by
    maximum-likelihood fitting of a $\mathrm{Beta}(\alpha,2\alpha)$ law to
    the simulated normalized effective rank variance $s(\sigma)/3$.}
    \label{fig:alpha-mle}
\end{figure}

\brem
Proposition \ref{prop:binary-spin} provides one exact beta-copula example, corresponding to $\alpha=1/2$. A natural open problem, and a possible direction for future work, is to determine whether there exists a broader family of scale-mixture models for which the normalized effective rank variance follows exactly a beta distribution with parameters $\alpha$ and $2\alpha$, for a nontrivial range of values of $\alpha$.
\erem

\FloatBarrier
\subsubsection{One-cut approximation}
Another natural solvable class is obtained by assuming that the density of $s(\sigma)/3$ has the more general one-cut form $\rho(x)=\sqrt{x(1-x)}\,Q(x)\mathbf 1_{[0,1]}(x),$ where $Q$ is a polynomial. In that case, the Stieltjes transform of $\rho$ is explicit according to Proposition \ref{propo:one_cut}. Moreover, it follows from \eqref{eq:r_transform} that $R_{F_{q, \mu_s}}\!(w)=3z^2m_\rho(z) - 3z$ where $z=(3qw)^{-1}$.

\begin{propo}\la{propo:one_cut}
Let $a<b$, let $Q(x)=\sum_{k=0}^D q_k x^k$ be a non-zero polynomial of degree $D\geq0$, and let
$$
\rho(x)
= Q(x)\sqrt{(x-a)(b-x)}\,\mathbf 1_{[a,b]}(x)
$$
be a probability density. Let $\Delta(z)=\sqrt{(z-a)(z-b)}$ be the branch such that $\Delta(z)\sim z$ as $z\to\infty$. Then, for $z\in \C\setminus [a,b],$
$$
m_\rho(z)
=
\pi Q(z)\left(z-\frac{a+b}{2}-\Delta(z)\right) - \Pi_Q(z),
$$
where $\Pi_Q$ is the polynomial
$$
\Pi_Q(z)
=
\sum_{k=1}^D q_k
\sum_{\ell=0}^{k-1}
z^{k-1-\ell}M_\ell,
\qquad
M_\ell
=
\int_a^b x^\ell\sqrt{(x-a)(b-x)}\,dx.
$$
\end{propo}
\begin{proof}
This is a standard one-cut structure that appears in the computation of log-gas equilibrium densities; see, for instance, \cite[Chap.\ 5]{potters2020first}.
\end{proof}

Compared with the one-parameter beta approximation, this class has additional degrees of freedom through the polynomial $Q$. Indeed, the beta approximation $\mathrm{Beta}(\alpha,2\alpha)$ has a prescribed edge behaviour: its density is proportional to $x^{\alpha-1}$ as $x\downarrow 0$ and to $(1-x)^{2\alpha-1}$ as $x\uparrow 1$. By contrast, $Q$ provides additional edge flexibility: if it has zeros of multiplicities $m_0,m_1\in\mathbb{N}$ at $0$ and $1$, respectively, then
$$
\rho(x)\asymp x^{m_0+1/2}
\quad\text{as }x\downarrow 0,
\qquad
\rho(x)\asymp (1-x)^{m_1+1/2}
\quad\text{as }x\uparrow 1.
$$
We therefore consider $Q(x)=x^{m_0}(1-x)^{m_1}P(x),$ where the multiplicities $(m_0,m_1)$ are fitted together with the coefficients of the remaining polynomial factor $P,$ with $P(0)P(1)\ne 0$. The choice of $P$ is guided by the observed shape of the empirical density of $s(\sigma)/3$. In our numerical experiments with Student data, Fig.\ \ref{fig:fit_beta} shows that this density concentrates around $1/3$, so we use a polynomial correction centered at $1/3$, namely
$$
P(x)
=
\beta_0 \left(
1+\sum_{k=1}^K \beta_k\left(x-\frac13\right)^k
\right)^2,
$$
where the coefficients $(\beta_k)_{0\le k \le K}$ are given and $K\ge0.$ The square guarantees nonnegativity of the fitted density, while the coefficients $(\beta_k)_{1\le k \le K}$ provide flexibility around the concentration point. This gives the parametric family
$$
\rho(x)
=
\beta_0 \, x^{m_0+1/2}(1-x)^{m_1+1/2}
\left(
1+\sum_{k=1}^K \beta_k\left(x-\frac13\right)^k
\right)^2
\mathbf 1_{[0,1]}(x),
$$
which remains explicitly solvable through Proposition \ref{propo:one_cut}. Fig.\ \ref{fig:fit_beta} shows that the $\mathrm{Beta}(\alpha,2\alpha)$ approximation is already reasonably accurate across all values of $\nu$, and becomes very good for larger degrees of freedom, especially for $\nu\geq 4$, as the density of $s(\sigma)/3$ concentrates around $1/3$. The polynomial one-cut approximation provides an accurate fit across the whole range of degrees of freedom. Its advantage is particularly visible for small $\nu$, where the beta approximation has more difficulty capturing the sharp mode and the asymmetric tail behavior. In practice, since the empirical spectral distribution obtained with the beta-copula approximation already provides a very good fit, as observed in Fig.\ \ref{fig:student}, the additional model complexity of the one-cut fit may not always be worth the computational cost.

\begin{figure}[h!]
    \centering
    \includegraphics[width=0.9\textwidth]{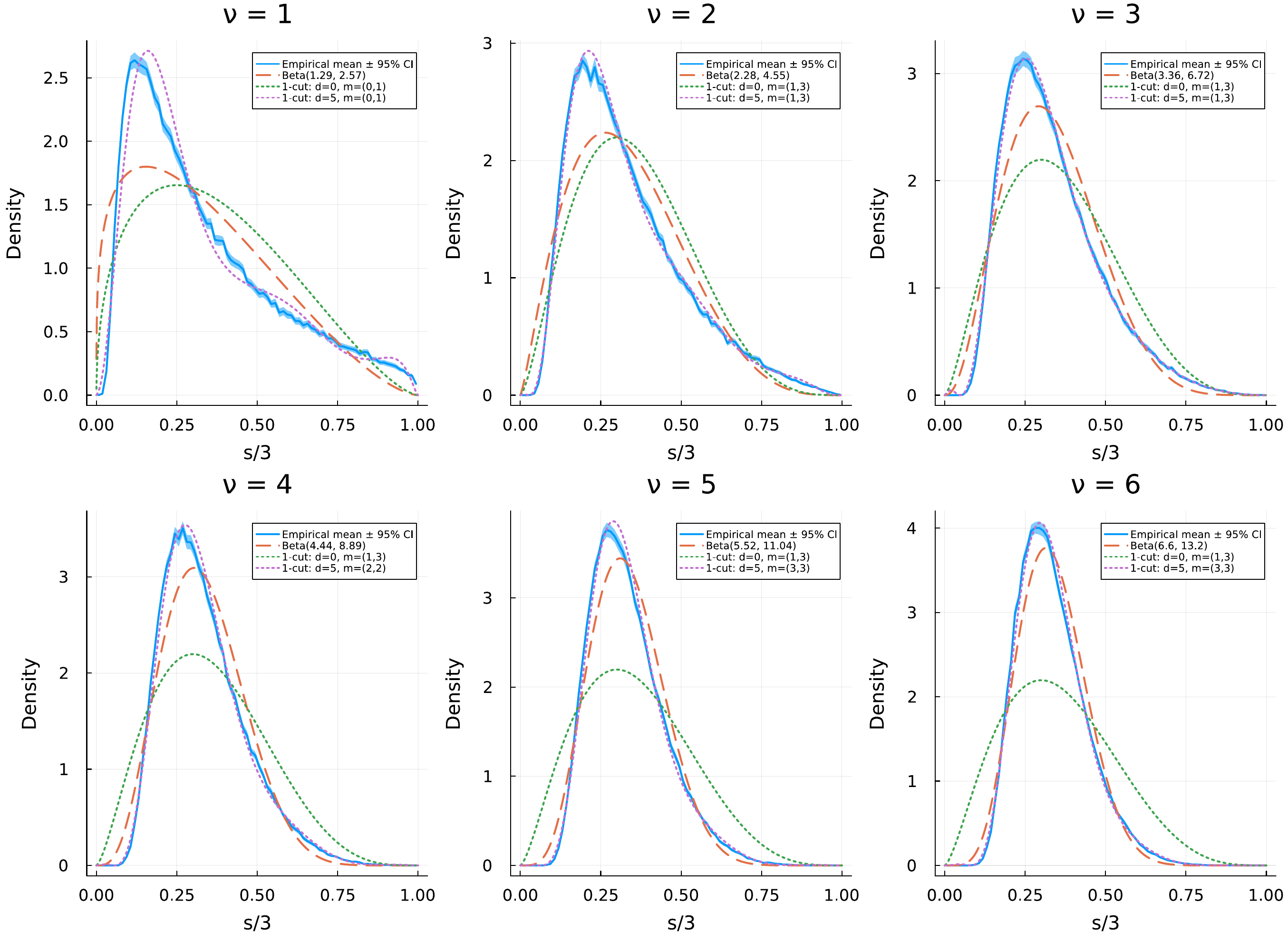}
    \caption{Fit of the empirical density of $s(\sigma)/3$ for Student data with degrees of freedom $\nu=1,\ldots,6$. The solid blue curve is the average empirical density over 100 simulations, with the shaded region corresponding to a 95\% confidence interval. The dashed orange curve is the fitted $\mathrm{Beta}(\alpha,2\alpha)$ approximation. The dotted curves correspond to the polynomial one-cut approximations for $K\in\{0,5\}$ and multiplicities $m=(m_0, m_1)$.}
    \label{fig:fit_beta}
\end{figure}

\FloatBarrier
\section{Latent-Variable Dependence Model}\la{sec:gene_model}
We now show that the mechanism behind Theorem \ref{thm:main} is not specific to pure scale mixtures. Although the scale-mixture model remains the main model of interest, it is useful to formulate the argument in a broader latent-variable framework, where the same spectral mechanism appears in a more abstract form. In this setting, the dependence between coordinates is generated by a common latent variable $L_t$. The scale-mixture model of Section \ref{sec:model-result} is recovered as a special case, while the generalized framework also covers models with correlated directional components, such as one-factor Gaussian scale mixtures.

\subsection{Model and limiting spectral distribution}
Let $m\ge 1$ and $H:\R^m\times\R\to\R$ be a measurable function. For each $T \geq 1$, let $N=N(T)$ and assume that:
\begin{enumerate}
    \item[(a)] $(L_t)_{1\leq t\leq T}$ are i.i.d.\ copies of a multivariate random variable $L\in\R^m$, 
    \item[(b)] $(\eps_{t,n})_{1\leq t\leq T,\, 1\leq n\leq N}$ are i.i.d.\ copies of a real random variable $\eps$,
    \item[(c)] the families $(L_t)_t$ and $(\eps_{t,n})_{t,n}$ are independent,
    \item[(d)] the distribution function $F$ of $H(L, \eps)$ is continuous. 
\end{enumerate}
For each $1\leq t \leq T$, we consider the observations \be\la{eq:generalized_model} 
x_{t,n} = H(L_t, \eps_{t,n}).
\ee
Throughout this section, we keep the notation $F$ for the common marginal distribution function of $H(L,\eps)$, and we define the associated standardized rank transform as before by $g(x)=\sqrt{3}\bigl(2F(x)-1\bigr).$ For $l\in\R^m$ consider $$\phi(l)=\E[g(H(l,\eps))] \qquad \text{and} \qquad \psi(l)=\Var[g(H(l,\eps))].$$
We denote by $\mu_\psi$ the law of $\psi(L)$, i.e., $$\mu_\psi=\operatorname{Law}(\psi(L)),$$ which is supported on $[0,3],$ since $|g|\leq \sqrt{3}$.

\bTh
\la{thm:generalized_main}
Suppose that {\rm (a)--(d)} hold. Assume moreover that $\frac{N}{T}\rightarrow q\in(0,\infty)$ and that $\psi(L)>0$ almost surely. Then the empirical spectral distribution of the sample Spearman rank correlation matrix $R_N$ converges weakly almost surely to a deterministic probability measure $F_{q,\mu_\psi}$ on $[0,\infty)$. Moreover, its Stieltjes transform $m(\cdot)$ is the unique function $m:\C^+\to \C^-$ satisfying
\be
z = \frac{1}{m(z)} + \int_{[0,3]}
\frac{x}{1-qxm(z)}\,d\mu_\psi(x),
\qquad z\in\C^+.
\ee
\eTh

\brem\la{rem:outlier_eigval}
Although Theorem \ref{thm:generalized_main} has the same limiting bulk description as Theorem \ref{thm:main}, the generalized setting contains an additional finite-rank component. More precisely, in the proof we decompose the oracle-score covariance matrix as in \eqref{eq:decomp_gene_terms}. The finite-rank terms in this decomposition, i.e., the last three terms on the right-hand side, do not affect the weak empirical spectral distribution. However, one of them may create a large isolated eigenvalue. Indeed, the term $\frac{\|\phi_T\|^2}{T}\mathsf{1}_N\mathsf{1}_N^\top$ has one nonzero eigenvalue equal to $N\frac{\|\phi_T\|^2}{T}$, which is of order $N\E[\phi(L)^2]$, where $\phi_T=(\phi(L_1),\ldots,\phi(L_T))^\top$.
\erem

\begin{cor}\la{coro:gene-moment-proof}
Let $F_{q,\mu_\psi}$ be the LSD in Theorem \ref{thm:generalized_main}. Then $F_{q,\mu_\psi}$ is compactly supported and has an atom at zero of mass $\left(1-\frac1q\right)_+$. Its first two moments are
$$
    \int x\,dF_{q,\mu_\psi}(x)=\E[\psi(L)]
$$
and
$$
    \int x^2\,dF_{q,\mu_\psi}(x)
    =
    \E[\psi(L)]^2+q\E[\psi(L)^2].
$$
Moreover, the free cumulants are given by $\kappa_k = q^{k-1}\E[\psi(L)^k]$ for $ k\geq 1.$
\end{cor}

\brem
The first moment in Corollary \ref{coro:gene-moment-proof} is in general smaller than one. Indeed, if
$Y=g(H(L,\eps))$, then $\E[Y]=0$ and $\E[Y^2]=1$, since $F(H(L,\eps))$ is uniform on $(0,1)$. By the total variance formula,
$$
1
=
\Var(Y)
=
\E[\Var(Y\mid L)]
+
\Var(\E[Y\mid L])
=
\E[\psi(L)]+\E[\phi(L)^2],
$$
where we used $\E[\phi(L)]=\E[Y]=0$. Hence
$$
\int x\,dF_{q,\mu_\psi}(x)
=
\E[\psi(L)]
=
1-\E[\phi(L)^2]
\leq 1.
$$
This does not contradict the fact that the Spearman matrix has normalized trace equal to one. The weak empirical spectral distribution only captures the limiting bulk. When $\phi$ is not identically zero, the missing first moment is carried by the finite-rank component discussed in Remark \ref{rem:outlier_eigval}, through one or finitely many eigenvalues of order $N$, whose contribution to the normalized trace is of order one, while their total mass in the empirical spectral distribution is $O(1/N)$.
\erem

\bpr
The proof is identical to that of Corollary \ref{coro:moment-proof}, with $\mu_s$ replaced by $\mu_\psi$.
\epr

\subsection{Models covered by the theorem}

\subsubsection{Scale-mixture model}
The scale-mixture model of Section \ref{sec:model-result} is recovered directly from \eqref{eq:generalized_model}. Indeed, take $m=1$, $L=\sigma$, $\eps=\xi$, and $H(\sigma,\xi)=\sigma\xi$. Then the data generating process becomes
$$
x_{t,n}=H(\sigma_t,\xi_{t,n})=\sigma_t\xi_{t,n},
$$
which is exactly \eqref{eq:scale-mixture-model}. In this case, the marginal distribution of $\sigma\xi$ is symmetric, hence the population rank transform $g$ is odd. Therefore, $\phi(\sigma)=\E[g(\sigma\xi)]=0.$ Moreover,
$$
\psi(\sigma)
=
\Var[g(\sigma\xi)]
=
\E[g(\sigma\xi)^2]
=
s(\sigma).
$$
Thus $\mu_\psi=\mu_s$, and Theorem \ref{thm:generalized_main} reduces to Theorem \ref{thm:main}. In particular, since $\phi(\sigma)=0,$ the finite-rank component discussed above vanishes and no large isolated eigenvalue is produced by the conditional mean.

The same framework also allows for a common drift, or market-direction, component. For instance, let $(\mu_t)$ be independent of $(\sigma_t)$ and $(\xi_{t,n})$, take $L_t=(\mu_t,\sigma_t)$, and define $H((\mu,\sigma),\xi)=\mu+\sigma\xi$. Then
$$
x_{t,n}=\mu_t+\sigma_t\xi_{t,n}.
$$
In this case, the limiting bulk is governed by the law of $\psi(\mu,\sigma)=\Var[g(\mu+\sigma\xi)]$, while a nonzero value of $\phi(\mu,\sigma)=\E[g(\mu+\sigma\xi)]$ may generate an additional large isolated eigenvalue without changing the limiting empirical spectral distribution.

\subsubsection{One-factor Gaussian model}

Let $\rho\in[0,1)$. Consider independent families $(\sigma_t)_t$, $(\eta_t)_t$ and $(\eps_{t,n})_{t,n}$, where $(\sigma_t)_t$ are i.i.d.\ positive random variables, $(\eta_t)_t$ are i.i.d.\ standard Gaussian variables, and $(\eps_{t,n})_{t,n}$ are i.i.d.\ standard Gaussian variables. Take $L_t=(\sigma_t,\eta_t)$ and define
$H((\sigma,\eta),\eps)=\sigma(\sqrt{\rho}\eta+\sqrt{1-\rho}\eps)$. Then
$$
x_{t,n}
=
\sigma_t(\sqrt{\rho}\eta_t+\sqrt{1-\rho}\eps_{t,n}).
$$
Conditionally on $L_t=(\sigma_t,\eta_t)$, the coordinates are independent, and the limiting bulk is governed by the law of
$$
\psi(\sigma,\eta)
=
\Var\!\left[g\!\left(\sigma(\sqrt{\rho}\eta+\sqrt{1-\rho}\eps)\right)\right],
$$
where the variance is taken only with respect to $\eps$. When $\rho=0$, this reduces to the scale-mixture model with Gaussian directional components, and $\operatorname{Law}(\psi(\sigma,\eta))=\operatorname{Law}(s(\sigma))$.

This model can also be viewed as a scale-mixture model with correlated Gaussian directions. Indeed, setting
$$
\xi_t=\sqrt{\rho}\eta_t\mathsf{1}_N+\sqrt{1-\rho}\eps_t,
\qquad
\eps_t=(\eps_{t,1},\ldots,\eps_{t,N})^\top,
$$
we have $\xi_t\sim \mathcal N(0,\Sigma_\rho)$, where
$$
\Sigma_\rho=(1-\rho)I_N+\rho\,\mathsf{1}_N\mathsf{1}_N^\top
$$
is the equicorrelation matrix. Hence $x_t=\sigma_t\xi_t$. In particular, if $\sigma_t^2\sim \operatorname{InvGamma}(\nu/2,\nu/2)$ for some $\nu>0$, independently of $\xi_t$, then $x_t$ has a multivariate Student distribution with $\nu$ degrees of freedom and scale matrix $\Sigma_\rho$.

The case $\rho=1$ is degenerate from the point of view of Theorem \ref{thm:generalized_main}, since the conditional variance $\psi(L)$ vanishes. Non-degenerate extensions with several common factors, or with a time-varying correlation parameter $\rho_t$, can be handled in the same way.

\subsubsection{Regime-switching model}

Let $K\ge 2$, and let $J_t\in\{1,\ldots,K\}$ be an i.i.d.\ regime variable such that $\P(J_t=k)=p_k$, with $p_k>0$ and $\sum_{k=1}^K p_k=1$. Take $L_t=J_t$, and let $(\eps_{t,n})_{t,n}$ be i.i.d.\ copies of a real random variable $\eps$, independent of $(J_t)_t$. Given parameters $\mu_1,\ldots,\mu_K\in\R$ and $\sigma_1,\ldots,\sigma_K>0$, define
$H(k,\eps)=\mu_k+\sigma_k\eps$. Then
$$
x_{t,n}=\mu_{J_t}+\sigma_{J_t}\eps_{t,n}.
$$
Thus, at each time $t$, all coordinates share the same regime $J_t$, while remaining conditionally independent given this regime. In this model, the limiting bulk is governed by the discrete measure
$$
\mu_\psi=\sum_{k=1}^K p_k \delta_{\psi_k},
\qquad
\psi_k=\Var[g(\mu_k+\sigma_k\eps)].
$$
Assuming $\psi_k>0$ for every $k$, Theorem \ref{thm:generalized_main} gives the limiting spectral distribution through the equation
$$z = \frac{1}{m(z)}+ \sum_{k=1}^K p_k
\frac{\psi_k}{1-q\psi_k m(z)}.
$$

\subsubsection{Time-varying tail-thickness model}

The latent-variable framework can also describe models in which the tail thickness changes over time. Let $L_t=(\sigma_t,\nu_t)$, where $\sigma_t>0$ is a scale parameter and $\nu_t>0$ is a time-varying degrees-of-freedom parameter. Let $(\eps_{t,n})_{t,n}$ be i.i.d.\ uniform random variables on $(0,1)$, independent of $(L_t)_t$, and define
$$
H((\sigma,\nu),\eps)=\sigma F_\nu^{-1}(\eps),
$$
where $F_\nu$ denotes the distribution function of a Student $t_\nu$ random variable. Then
$$
x_{t,n}=\sigma_t F_{\nu_t}^{-1}(\eps_{t,n}).
$$
Conditionally on $L_t=(\sigma_t,\nu_t)$, the coordinates are i.i.d.\ with distribution $\sigma_t t_{\nu_t}$. The limiting bulk is therefore governed by the law of
$$ \psi(\sigma,\nu) = \Var[g(\sigma F_\nu^{-1}(\eps))].$$
This model allows both the volatility level and the tail thickness of the cross-sectional distribution to vary over time.


\section{Proofs of Theorems \ref{thm:main} and \ref{thm:generalized_main}}\la{sec:proof}

\subsection{Proof of Theorem \ref{thm:main}} 
The proof proceeds in three steps. We first replace the ranks $Q_{n,t}$ by the oracle scores $g(x_{t,n})$, whose sample covariance has a generalized Marčenko--Pastur limit $F_{q,\mu_s}$ after conditioning on the mixture variable $\sigma$. We then pass from oracle scores to empirical cumulative distribution function (ECDF) scores via the Dvoretzky--Kiefer--Wolfowitz (DKW) inequality, following the approach of \cite{wu2022limiting}. Finally, we show that the exact rank normalization defining $R_N$ differs from the ECDF matrix by a vanishing perturbation. We shall use the following notation throughout the proof. We write $\Rightarrow$ for weak convergence of probability measures and denote by $L(\cdot,\cdot)$ the Lévy distance between distribution functions. For an $N\times N$ Hermitian matrix $A$, we write $\|A\|$ for its spectral norm and, for $z\in\C^+,$ $m_A(z)=\frac1N\Tr(zI_N-A)^{-1}$ its Stieltjes transform.

\medskip

Define the oracle-score matrix $V\in\R^{N\times T}$ by $V_{n,t}=g(x_{t,n})$ and set
\be
C_N^{\mathrm{or}}
=
\frac{1}{T}VV^\top.
\la{eq:proof-oracle-matrix}
\ee
Next, for each coordinate $n$, let
$$
\widehat F_n(x)
=
\frac{1}{T}\sum_{t=1}^T \mathbf 1_{\{x_{t,n}\leq x\}},
\qquad x\in\R,
$$
and define the ECDF scores $ \widehat V_{n,t} = \sqrt{3}\bigl(2\widehat F_n(x_{t,n})-1\bigr).$
Writing $\widehat V=(\widehat V_{n,t}) \in \R^{N\times T}$, set
$$
\widehat C_N
=
\frac{1}{T}\widehat V\widehat V^\top.
$$
Heuristically, the proof of Theorem \ref{thm:main} is then organized as follows:
$$
C_N^{\mathrm{or}}
\quad \longrightarrow \quad
\widehat C_N
\quad \longrightarrow \quad
R_N \quad \longrightarrow \quad F_{q,\mu_s}.
$$
More precisely, we prove in Proposition \ref{prop:oracle-score-lsd} that $F^{C_N^{\mathrm{or}}}
\Rightarrow
F_{q,\mu_s}$ almost surely, then show in Proposition \ref{prop:ecdf-pop-comparison} that $L\!\left(F^{C_N^{\mathrm{or}}},F^{\widehat C_N}\right)
\longrightarrow 0$ almost surely, and finally prove in Proposition \ref{prop:true_spearman_same_lsd} that $\left\|R_N-\widehat C_N\right\|
\longrightarrow 0$ almost surely.
Now fix $z\in\C^+$. Using the resolvent identity,
\begin{align}
\bigl|m_{R_N}(z)-m_{\widehat C_N}(z)\bigr|
&=
\frac1N\Bigl|
\Tr\Bigl(
(zI_N-R_N)^{-1}
(R_N-\widehat C_N)
(zI_N-\widehat C_N)^{-1}
\Bigr)
\Bigr|
\nonumber\\
&\le
\|(zI_N-R_N)^{-1}\|\,
\|R_N-\widehat C_N\|\,
\|(zI_N-\widehat C_N)^{-1}\|
\nonumber\\
&\le
\frac{\|R_N-\widehat C_N\|}{(\Im z)^2} \longrightarrow 0 \quad \text{a.s.}\nonumber
\end{align}
Together, these results imply $F^{R_N}\Rightarrow F_{q,\mu_s}$ almost surely, which proves Theorem \ref{thm:main}.

\begin{proposition}
\la{prop:oracle-score-lsd}
Assume the hypotheses of Theorem \ref{thm:main}. Then the empirical spectral
distribution of the oracle-score matrix $C_N^{\mathrm{or}}$ converges weakly almost surely to a deterministic probability measure
$F_{q,\mu_s}$ on $[0,\infty)$.
Moreover, its Stieltjes transform $m(\cdot)$ is the unique function $m:\C^+\to \C^-$ satisfying
\be
z=
\frac{1}{m(z)}
+
\int_{[0,3]}
\frac{x}{1-qxm(z)}\,d\mu_s(x),
\qquad z\in\C^+.
\la{eq:oracle-score-fixed-point}
\ee
\end{proposition}

\bpr
Let $(\Omega, \mathcal{F}, \mathbb{P})$ be a probability space. Set $s_t=s(\sigma_t)$ and $D_T=\diag(s_1,\dots,s_T).$ When $s_t>0$ for all $1\leq t \leq T$, define $w_{t,n}=\frac{g(x_{t,n})}{\sqrt{s_t}}$ and $W_T=(w_{t,n})_{1\le t\le T,\ 1\le n\le N}\in\R^{T\times N}.$
Then, \eqref{eq:proof-oracle-matrix} becomes
$$
C_N^{\rm or}=\frac1T W_T^\top D_T W_T.
$$
Let $\mathcal G=\sigma(\sigma_1,\sigma_2,\dots)$ be the $\sigma$-algebra generated by the volatility sequence $(\sigma_t)_{t\geq1}$. Conditionally on $\mathcal G$, the random variables $(w_{t,n})_{t,n}$ are independent. Moreover, by Remark \ref{rmk:condi_spearman_moments} for the first identity and by the definition of $s_t$ for the second, we have
$$
\E[w_{t,n}\mid \mathcal G]=0,
\qquad
\E[w_{t,n}^2\mid \mathcal G]=1
\qquad\text{for all }t,n.
$$
Because the $s_t$ are i.i.d.\ with common law $\mu_s$, the empirical distribution function $\mu_T^s$ of $\{s_1,\ldots,s_T\}$ converges weakly to $\mu_s$ almost surely.
Let
$$
\Omega_0=
\Big\{\mu_T^s \Rightarrow \mu_s\Big\}
\cap
\Big\{s_t>0\text{ for all }t\ge1\Big\}.
$$ Then $\Omega_0 \in \mathcal{G}$ and $\P(\Omega_0)=1$. For fixed $\omega \in \Omega_0$, we work under the conditional probability $\mathbb{P}_{\omega}(\cdot) = \mathbb{P}(\cdot|\mathcal{G})(\omega)$ and denote by $\E_\omega$ the corresponding expectation. Under $\P_\omega,$ we verify that the assumptions of Lemma \ref{lem:generalized-mp} are satisfied: the matrix $D_T$ is deterministic, its empirical spectral distribution is $\mu_T^s(\omega)=\frac1T\sum_{t=1}^T \delta_{s_t(\omega)}$ and $\mu_T^s(\omega)\Rightarrow \mu_s$ which has compact support on $[0,3]$. It remains to verify the Lindeberg condition
\eqref{eq:linderberg}. For $\eta>0$, define
$$
L_T^{\omega}(\eta)=
\frac1{NT}\sum_{t=1}^T\sum_{n=1}^N
\E_{\omega}\!\!\Big[w_{t,n}^2\,\mathbf 1_{\{|w_{t,n}|>\eta\sqrt N\}}\Big].
$$
Since, conditionally on $\mathcal G$, the $N$ variables in each row have the same law,
$$
L_T^{\omega}(\eta)
=
\frac1T\sum_{t=1}^T
\E_{\omega}\!\!\Big[w_{t,1}^2\,\mathbf 1_{\{|w_{t,1}|>\eta\sqrt N\}}\Big].
$$
Since $|g| \le \sqrt3$, we have $|w_{t,n}|^2\le \frac{3}{s_t}$ for all $t,n$.
Thus, on the event $\{|w_{t,1}|>\eta\sqrt N\}$ we must have $\frac{3}{s_t(\omega)}\ge |w_{t,1}|^2>\eta^2 N,$ hence $\{|w_{t,1}|>\eta\sqrt N\}\subset\Big\{s_t(\omega)<\frac{3}{\eta^2 N}\Big\}.$
Therefore, since $\E_{\omega}[w_{t,1}^2]=1$,
$$
L_T^{\omega}(\eta)
\le
\frac1T\sum_{t=1}^T \mathbf 1_{\{s_t(\omega)<3/(\eta^2 N)\}}.
$$
Fix $\delta>0$. For all large
enough $T$, we have $3/(\eta^2 N)<\delta$, and thus
$$
\frac1T\sum_{t=1}^T \mathbf 1_{\{s_t(\omega)<3/(\eta^2 N)\}}
\le
\frac1T\sum_{t=1}^T \mathbf 1_{\{s_t(\omega)\le \delta\}}
=
\mu_T^s(\omega)([0,\delta]).
$$
Using Portmanteau's theorem for the closed set $[0,\delta]$, $\limsup_{T\to\infty}\mu_T^s(\omega)([0,\delta])
\le
\mu_s([0,\delta]).$
Since $\delta>0$ is arbitrary and $\mu_s(\{0\})=0$ by assumption, letting $\delta\downarrow0$ gives, by continuity from above,
$$
L_T^{\omega}(\eta)\longrightarrow0
\qquad\text{for every }\eta>0.
$$
Therefore, all the assumptions of Lemma \ref{lem:generalized-mp} are satisfied under $\P_{\om}$. Since the limiting law is independent of $\omega \in \Omega_0$ and $\mathbb{P}(\Omega_0)=1$, this proves the result of the proposition.
\epr

\medskip
\begin{proposition}\la{prop:ecdf-pop-comparison}
Assume the hypotheses of Theorem \ref{thm:main}.
Then
$$ 
    L\!\left(F^{C_N^{\mathrm{or}}},F^{\widehat C_N}\right) \longrightarrow 0 \qquad \text{a.s.}
$$
\end{proposition}

\bpr
For each fixed $n$, the random variables $x_{1,n},\dots,x_{T,n}$ are i.i.d. with distribution function $F$. Hence, by the DKW inequality, for every $\varepsilon>0$,
$$
    \mathbb P\left(
        \sup_{x\in\mathbb R}
        \left|\widehat F_n(x)-F(x)\right|
        >
        \varepsilon
    \right)
    \leq
    2e^{-2T\varepsilon^2}.
$$
Taking a union bound over $n=1,\dots,N$ gives
$$
    \mathbb P\left(
        \max_{1\leq n\leq N}
        \sup_{x\in\mathbb R}
        \left|\widehat F_n(x)-F(x)\right|
        >
        \varepsilon
    \right)
    \leq
    2Ne^{-2T\varepsilon^2}.
$$
Since $N/T\to q$, the right-hand side is summable in $T$. Therefore, by the Borel--Cantelli lemma,
$$
    \Delta_T
    =
    \max_{1\leq n\leq N}
    \sup_{x\in\mathbb R}
    \left|\widehat F_n(x)-F(x)\right|
    \longrightarrow 0
    \qquad \text{a.s.}
$$
Consequently,
$$
    \max_{\substack{1\leq n\leq N\\ 1\leq t\leq T}}
    \left|
        \widehat V_{n,t}-V_{n,t}
    \right|
    \leq
    2\sqrt{3}\,\Delta_T
    \longrightarrow 0
    \qquad \text{a.s.}
$$
Now set $Y_N=\frac{1}{\sqrt T}V$ and $\widehat Y_N=\frac{1}{\sqrt T}\widehat V$ such that $ C_N^{\mathrm{or}} = Y_NY_N^\top$ and $\widehat C_N = \widehat Y_N\widehat Y_N^\top.$
Applying Lemma \ref{coro:coro-a42} with $A=Y_N$ and $B=\widehat Y_N$, we obtain
$$
    L^4\left(
        F^{C_N^{\mathrm{or}}},
        F^{\widehat C_N}
    \right)
    \leq
    \frac{2}{N^2}
    \Tr\left(Y_NY_N^\top+\widehat Y_N\widehat Y_N^\top\right)
    \Tr\left((Y_N-\widehat Y_N)(Y_N-\widehat Y_N)^\top\right).
$$
Since $|V_{n,t}|\leq \sqrt{3}$ and $|\widehat V_{n,t}|\leq \sqrt{3},$
we have
$$
    \frac{1}{N}\Tr(Y_NY_N^\top)
    =
    \frac{1}{NT}\sum_{n=1}^N\sum_{t=1}^T |V_{n,t}|^2
    \leq 3
$$
and we get the same upper bound for $\frac{1}{N}\Tr(\widehat Y_N\widehat Y_N^\top).$ Moreover,
$$
\begin{aligned}
    \frac{1}{N}
    \Tr\left((Y_N-\widehat Y_N)(Y_N-\widehat Y_N)^\top\right)
    &=
    \frac{1}{NT}
    \sum_{n=1}^N\sum_{t=1}^T
    |V_{n,t}-\widehat V_{n,t}|^2 \\
    &\leq
    12\Delta_T^2
    \longrightarrow 0
    \qquad \text{a.s.}
\end{aligned}
$$
This proves the claim.
\epr

\medskip

\begin{proposition}\la{prop:true_spearman_same_lsd}
Assume the hypotheses of Theorem \ref{thm:main}. Then 
$$\left\|R_N-\widehat C_N\right\|
\longrightarrow 0
\qquad \text{almost surely}.$$
\end{proposition}
\bpr
Since $F$ is continuous, ties occur with probability zero. By a countable intersection argument, there are no ties in the data on a probability-one event. Note that for every $n,t$,
$$
Z_{n,t}
=
\frac{T}{\sqrt{T^2-1}}\frac{(\widehat V)_{n,t}}{\sqrt T}
-\frac{\sqrt3}{\sqrt{T(T^2-1)}}.
$$
Hence, if we write $ \mathbf 1_N=(1,\dots,1)^\top\in\R^N$ and $\mathbf 1_T=(1,\dots,1)^\top\in\R^T,$
then
\be
Z
=
\frac{T}{\sqrt{T^2-1}}\,\widehat Y_N
-\frac{\sqrt3}{\sqrt{T(T^2-1)}}\,\mathbf 1_N\mathbf 1_T^\top.
\la{eq:matrix_identity_Z}
\ee
In the absence of ties, $Q_{n,1},\dots,Q_{n,T}$ is a permutation of $\{1,\dots,T\}$ so that,
$$
\sum_{t=1}^T (\widehat V)_{n,t}
=
\sqrt3\sum_{t=1}^T\Bigl(\frac{2Q_{n,t}}{T}-1\Bigr)
=
\sqrt3.
$$
Thus
\be
\widehat Y_N\mathbf 1_T=\sqrt{\frac3T}\,\mathbf 1_N.
\la{eq:U_times_one}
\ee
Now expand $R_N\!=\!ZZ^\top\!$ using \eqref{eq:matrix_identity_Z}. Since $(\mathbf 1_N\mathbf 1_T^\top)(\mathbf 1_N\mathbf 1_T^\top)^\top \!=
T\,\mathbf 1_N\mathbf 1_N^\top,$
and using \eqref{eq:U_times_one}, we obtain
\begin{align}
R_N
&=
\frac{T^2}{T^2-1}\widehat Y_N \widehat Y_N^\top
-\frac{2T}{\sqrt{T^2-1}}\frac{\sqrt3}{\sqrt{T(T^2-1)}}\,
\widehat Y_N\mathbf 1_T\mathbf 1_N^\top
+\frac{3}{T(T^2-1)}\,T\,\mathbf 1_N\mathbf 1_N^\top
\nonumber\\
&=
\frac{T^2}{T^2-1}\widehat C_N
-\frac{3}{T^2-1}\,\mathbf 1_N\mathbf 1_N^\top\nonumber.
\end{align}
Consequently,
\be
R_N-\widehat C_N
=
\frac1{T^2-1}\widehat C_N
-\frac{3}{T^2-1}\,\mathbf 1_N\mathbf 1_N^\top.
\la{eq:difference_identity}
\ee
Moreover, $
\|\widehat C_N\|
\le
\Tr(\widehat C_N)
\le 3N$ and 
$\|\mathbf 1_N\mathbf 1_N^\top\|=N.$
Therefore, by \eqref{eq:difference_identity},
$$
\|R_N-\widehat C_N\|
\le
\frac{3N}{T^2-1}+\frac{3N}{T^2-1}
=
\frac{6N}{T^2-1}.
$$
Since $N/T\to q\in(0,\infty)$, this yields the result.
\epr

\subsection{Proof of Theorem \ref{thm:generalized_main}}

We now explain how the proof of Theorem \ref{thm:main} adapts to the latent-variable setting. The argument follows the same three steps as before, with the only difference that the oracle-score matrix must first be centered conditionally on the latent variables.

Let $\mathcal G=\sigma(L_t:t\ge1)$ be the $\sigma$-algebra generated by the sequence $(L_t)_{t\geq1}$. Conditionally on $\mathcal{G},$ the variables $(V_{n,t})_{n,t}$ are independent. Moreover, $\E[V_{n,t}\mid\mathcal{G}]=\phi(L_t)$ and $\Var[V_{n,t}\mid\mathcal{G}]=\psi(L_t).$ Consider $A_{n,t}$ such that $V_{n,t}=A_{n,t} + \phi(L_t)$ for each fixed $n,t.$ Then, in matrix form, this becomes $V = A + \mathbf{1}_N\phi_T^\top$ where $\phi_T = (\phi(L_1),\ldots,\phi(L_T))^\top.$ Therefore, 
\be\la{eq:decomp_gene_terms}\frac1T VV^\top= \frac1T AA^\top + \frac1T A\phi_T \mathbf{1}_N^\top + \frac1T \mathbf{1}_N\phi_T^\top A^\top + \frac{\|\phi_T\|^2}{T}\mathbf{1}_N\mathbf{1}_N^\top.\ee
The sum of the last three terms on the right-hand side has rank at most three. Thus, by Lemma \ref{lem:rank_ineq}, $$\|F^{C_N^{\mathrm{or}}} - F^{T^{-1}A A^\top}\|_{\infty} \le \frac{3}{N} \longrightarrow 0,$$ which means $C_N^{\mathrm{or}}$ and $T^{-1}A A^\top$ have the same LSD. 

Set $\psi_t = \psi(L_t), \, D_T =\diag(\psi_1,\ldots,\psi_T)$ and $w_{t,n}=\frac{A_{n,t}}{\sqrt{\psi_t}}.$ Then,
$$\frac1TAA^\top=\frac1T W_T^\top D_T W_T$$ where $W_T=(w_{t,n})_{1\le t\le T,\ 1\le n\le N}\in\R^{T\times N}.$ Conditionally on $\mathcal{G},$ the entries of $W_T$ are independent, centered, and have variance one: 
$$\E[w_{t,n}\mid\mathcal{G}]=0, \qquad \E[w_{t,n}^2\mid\mathcal{G}]=1.$$ Since $|g|\le\sqrt{3},$ one has $|A_{n,t}|\le2\sqrt{3}$ and therefore $w_{t,n}^2\le \frac{12}{\psi_t}.$ Arguing as in the proof of Proposition \ref{prop:oracle-score-lsd}, one has, for every $\eta>0,$
$$\frac1{NT}\sum_{t=1}^T\sum_{n=1}^N
\E\!\!\Big[w_{t,n}^2\,\mathbf 1_{\{|w_{t,n}|>\eta\sqrt N\}} \mid \mathcal{G}\Big] \le \frac1T \sum_{t=1}^T\mathbf{1}_{\{\psi_t\le12/(\eta^2N)\}}.$$ The variables $(\psi_t)_t$ are i.i.d.\ with common law $\mu_{\psi}$ which satisfies $\mu_\psi(\{0\})=0$ since $\psi(L)>0$ almost surely. Thus, the right-hand side converges to zero almost surely, and the deterministic conditioning argument used in Proposition \ref{prop:oracle-score-lsd} shows that Lemma \ref{lem:generalized-mp} applies conditionally on $\mathcal G$. It follows that $$F^{T^{-1} AA^\top} \Longrightarrow F_{q,\mu_\psi},$$ and the same convergence holds for $C_N^{\mathrm{or}}$ by the rank inequality above. Moreover, the proofs of Propositions \ref{prop:ecdf-pop-comparison} and \ref{prop:true_spearman_same_lsd} apply without change, since they only use the common continuous marginal distribution and the boundedness of the rank scores. Therefore, $$F^{R_N} \Longrightarrow F_{q,\mu_\psi} \qquad \text{almost surely}.$$

 
\section{Appendix}\la{sec:appendix}

\subsection{Proof of Corollary \ref{coro:moment-proof}}\la{app:moment-proof}

The formula for the atom at the origin is a direct consequence of \cite{silverstein1995analysis}. We now prove $F_{q,\mu_s}$ is compactly supported.
Let $S_{q,\mu_s}\!=\operatorname{supp}(F_{q,\mu_s})$ and define
$$
x_q(m)=
\frac{1}{m}
+
\int_{[0,3]}\frac{s}{1-qsm}\,d\mu_s(s).
$$
By the support characterization of Silverstein and Choi
\cite[Theorems 4.1--4.2]{silverstein1995analysis}, applied to the fixed-point equation
of Theorem \ref{thm:main}, for every $x\neq 0$,
\be
x\in S_{q,\mu_s}^{c}
\quad\Longleftrightarrow\quad
x=x_q(m)
\text{ for some } m\in B_q \text{ such that } x_q'(m)<0,
\la{eq:support_charac}
\ee
where $B_q=\{m\in\R:\ m\neq 0,\ 1-qsm\neq 0
\text{ for all } s\in\operatorname{supp}(\mu_s)\}.$
Since $\operatorname{supp}(\mu_s)\subset[0,3]$, choose $\varepsilon_q=\frac{1}{6(q+\sqrt q)}.$
For $m\in(0,\varepsilon_q)$ and $s\in[0,3]$, we have $1-qsm \ge 1-3q\varepsilon_q>0.$ Hence $(0, \varepsilon_q)\subset B_q$. Moreover, $x_q$ is differentiable on $(0, \varepsilon_q)$ and
$$
x_q'(m)=
-\frac{1}{m^2}
+
q\int_{[0,3]}\frac{s^2}{(1-qsm)^2}\,d\mu_s(s).
$$
One readily proves that $x_q'(m)<0$ on $(0, \varepsilon_q)$. Therefore, \eqref{eq:support_charac} yields $x_q((0, \varepsilon_q))\subset S_{q,\mu_s}^{c}.$
Finally, $x_q$ is continuous and strictly decreasing on $(0,\varepsilon_q)$, and $\lim_{m\downarrow 0}x_q(m)=+\infty.$ Thus $x_q((0, \varepsilon_q))=(L_q,\infty)$ for some finite $L_q>0$. This concludes that $\operatorname{supp}(F_{q,\mu_s})\subset[0,L_q].$ Since $F_{q,\mu_s}$ is compactly supported, it has finite moments and finite free cumulants of all orders. Moreover, given that $0 \leq s \leq 3$, we can expand \eqref{eq:r_transform} for $w$ near 0, 
$$
R_{F_{q,\mu_s}}\!(w) = \sum_{k \geq 1} q^{k-1}\E[s(\sigma)^k]w^{k-1},
$$
which identifies the free cumulants $\kappa_k = q^{k-1}\E[s(\sigma)^k]$ for $k\geq 1$. The first two moments follow directly from the moment-cumulant formula.

\subsection{Proof of Proposition \ref{prop:binary-spin}}\la{app:proof-binary-spin}
Denote by $F_\sigma$ the distribution function of $\sigma$. Using the assumptions on $\xi$ and $\sigma$, the marginal distribution function $F$ of $x=\sigma\xi$ satisfies, for $a>0$,
$$
        F(a)
        =
        \mathbb P(\sigma\xi\leq a)
        =
        \frac12+\frac12\mathbb P(\sigma\leq a)
        =
        \frac12+\frac12 F_\sigma(a).
$$
Similarly, $F(-a)=\frac12-\frac12 F_\sigma(a).$
Therefore, $g(a)=\sqrt 3\,F_\sigma(a)$ and $g(-a)=-\sqrt 3\,F_\sigma(a).$
Hence
$$
        s(a)
        =
        \mathbb E\!\left[g(a\xi)^2\right]
        =
        3F_\sigma(a)^2.
$$
Since $F_\sigma(\sigma)$ is uniformly distributed on $(0,1)$, it follows that $s(\sigma)/3 \sim \textrm{Beta}(1/2,1).$
Finally, applying Theorem \ref{thm:main} and computing the integral \eqref{eq:r_transform} proves the result.

\subsection{Proof of Proposition \ref{prop:beta-copula}}\la{app:proof-beta-copula}
Let $X=s(\sigma)/3.$ By assumption, $X\sim \operatorname{Beta}(\alpha,2\alpha)$, so its density is
$$
        f_\alpha(x)
        =
        \frac{1}{B(\alpha,2\alpha)}
        x^{\alpha-1}(1-x)^{2\alpha-1},
        \qquad 0<x<1,
$$
where $B(\gamma, \eta) = \frac{\Gamma(\gamma)\Gamma(\eta)}{\Gamma(\gamma+\eta)}$ for $\gamma\!>\!0, \eta\!>\!0$ and $\Gamma$ is the Gamma function. First, the normalization is consistent with the fact that $g(x)$ has unit variance, 
$$
        \mathbb E[s(\sigma)]
        =
        3\mathbb E[X]
        =
        3\frac{\alpha}{\alpha+2\alpha}
        =
        1.
$$
Using the general formula for the $R$-transform given by Theorem \ref{thm:main},
$$
        R_{F_{q,\mu_s}}\!(w)=
        \frac{3}{B(\alpha,2\alpha)}
        \int_0^1
        \frac{x^\alpha(1-x)^{2\alpha-1}}{1-3qwx}
        \,dx.
$$
By Euler's integral representation of the Gauss hypergeometric function,
$$
        \int_0^1
        x^\alpha(1-x)^{2\alpha-1}(1-3qwx)^{-1}
        \,dx
        =
        B(\alpha+1,2\alpha)\,
        {}_2F_1\left(1,\alpha+1;3\alpha+1;3qw\right).
$$
Since $\frac{B(\alpha+1,2\alpha)}{B(\alpha,2\alpha)} = \frac13,$ this proves the claim.

\subsection{Useful Lemmas}
The following lemma is a standard generalized Marčenko--Pastur theorem. Although stated in Theorem 4.3 of \cite{bai2010spectral} with i.i.d.\ entries, its proof in Section 4.5 is carried out under the weaker Lindeberg condition \eqref{eq:linderberg}, which allows independent, non-identically distributed triangular arrays. Using the notation of their Theorem 4.3, we use this triangular-array form with $A_n=0$, $n=N$, $p=T$, $X_n=W_T^\top$, and $T_n=D_T$, followed by the rescaling from $1/N$ to $1/T$.
\blem
\la{lem:generalized-mp}
Let $N=N(T)$ be such that $\frac{N}{T}\rightarrow q\in(0,\infty).$ For each $T$, let $d_1^{(T)},\ldots,d_T^{(T)}\geq 0$ be deterministic $T$-dependent real numbers. Let $D_T=\diag\bigl(d_1^{(T)},\ldots,d_T^{(T)}\bigr)$ and assume that the empirical distribution function of $\{d_1^{(T)},\ldots,d_T^{(T)}\}$ converges to a compactly supported probability measure $\nu$ on $[0,\infty)$ as $T\to \infty$.
Let
$$
W_T=\bigl(w_{t,n}^{(T)}\bigr)_{1\leq t\leq T,\ 1\leq n\leq N}
\in \R^{T\times N}
$$
be a random matrix with independent real entries satisfying
$$
\E\bigl[w_{t,n}^{(T)}\bigr]=0,
\qquad
\E\bigl[(w_{t,n}^{(T)})^2\bigr]=1,
\qquad
1\leq t\leq T,\quad 1\leq n\leq N.
$$
Assume moreover that, for every $\eta>0$,
\be
\frac{1}{NT}
\sum_{t=1}^T\sum_{n=1}^N
\E\!\left[
(w_{t,n}^{(T)})^2
\mathbf 1_{\{|w_{t,n}^{(T)}|>\eta\sqrt{N}\}}
\right]
\longrightarrow 0.
\la{eq:linderberg}
\ee
Then the empirical spectral distribution of
$$
S_T=\frac{1}{T}W_T^\top D_T W_T
$$
converges almost surely to a deterministic probability measure $F_{q,\nu}$ on
$[0,\infty)$. Moreover, its Stieltjes transform $m(\cdot)$ is the unique function $m:\C^+\to\C^-$ satisfying
$$
z=
\frac{1}{m(z)}
+
\int_{[0,\infty)}
\frac{x}{1-qx\,m(z)}\,d\nu(x),
\qquad z\in\C^+.
$$
\elem

\blem[Corollary A.42 of \cite{bai2008large}]\la{coro:coro-a42}
For two $N\times T$ matrices
$A$ and $B$,
$$
    L^4\left(F^{AA^\top},F^{BB^\top}\right)
    \leq
    \frac{2}{N^2}
    \Tr\left(AA^\top+BB^\top\right)
    \Tr\left((A-B)(A-B)^\top\right).
$$
\elem

\blem[Theorem A.43 of \cite{bai2010spectral}] \label{lem:rank_ineq}
    Let $A$ and $B$ be two $N \times N$ Hermitian matrices. Then,
    \begin{equation}
        \|F^{A} - F^{B}\|_{\infty} \leq \frac{1}{N} \mathrm{rank}(A -B)
    \end{equation}
\elem


\section*{Acknowledgments}
We thank Jules Desperin and Iacopo Mastromatteo for useful discussions.


\bibliographystyle{plain}
\bibliography{sample}

\end{document}